\documentclass[letter,12pt]{article}

\usepackage[english]{babel}
\usepackage[utf8]{inputenc}

\usepackage[margin=1.2 in, top=1 in, bottom= 1.2 in]{geometry}

\makeatletter
\g@addto@macro\normalsize{%
  \setlength\abovedisplayskip{7pt}
  \setlength\belowdisplayskip{7pt}
  \setlength\abovedisplayshortskip{7pt}
  \setlength\belowdisplayshortskip{7pt}
}
\makeatother

\usepackage{tocloft}

\usepackage{amsfonts}
\usepackage{mathrsfs}
\usepackage{bbm}
\usepackage{latexsym}
\usepackage{math dots}
\usepackage{amssymb}
\usepackage{mathtools}
\usepackage{relsize}
\usepackage{lipsum}

\usepackage{todonotes}
\usepackage{enumitem}
\setlist{nolistsep} 	
\usepackage{amsthm}

\usepackage{xcolor}
\definecolor{Color1}{rgb}{0.0, 0.42, 0.47}
\definecolor{Color2}{rgb}{0.78, 0.11, 0.0}

\usepackage{titlesec}
\titleformat{\section}
  {\large\center\bfseries}
  {\thesection.}{.7em}{}
\titlespacing*{\section}{0pt}{3.5ex plus 0ex minus 0ex}{1.5ex plus 0ex}
\titleformat{\subsection}
  {\center\bfseries}
  {\thesubsection.}{.7em}{}
\titlespacing*{\subsection}{0pt}{3.5ex plus 0ex minus 0ex}{1.5ex plus 0ex}
\titleformat{\subsubsection}
  {\center\bfseries}
  {\thesubsubsection.}{.7em}{}
\titlespacing*{\subsubsection}{0pt}{3.5ex plus 0ex minus 0ex}{1.5ex plus 0ex}

\addto\captionsenglish{}

\usepackage{titling}
\makeatletter
\renewenvironment{abstract}{
\begin{center}
{\bfseries \large\abstractname\vspace{\z@}}
\end{center}
\quotation
}
\makeatother
\usepackage[linktocpage=true]{hyperref}
\usepackage{aliascnt}
\usepackage[capitalize]{cleveref}

\hypersetup{
    citecolor = Color2,
    colorlinks,
    linkcolor = black,
    urlcolor = Color2
}

\usepackage{aliascnt}

\theoremstyle{plain}

\newtheorem{Theorem}{Theorem}[section]

\newaliascnt{Fact}{Theorem}

\aliascntresetthe{Fact}

\newaliascnt{Lemma}{Theorem}
\newtheorem{Lemma}[Lemma]{Lemma}
\aliascntresetthe{Lemma}

\newaliascnt{Proposition}{Theorem}
\newtheorem{Proposition}[Proposition]{Proposition}
\aliascntresetthe{Proposition}

\newaliascnt{Corollary}{Theorem}
\newtheorem{Corollary}[Corollary]{Corollary}
\aliascntresetthe{Corollary}

\newaliascnt{Conjecture}{Theorem}

\aliascntresetthe{Conjecture}

\newaliascnt{Question}{Theorem}

\aliascntresetthe{Question}

\theoremstyle{definition}

\newaliascnt{Definition}{Theorem}
\newtheorem{Definition}[Definition]{Definition}
\aliascntresetthe{Definition}

\newaliascnt{Remark}{Theorem}
\newtheorem{Remark}[Remark]{Remark}
\aliascntresetthe{Remark}

\newaliascnt{Example}{Theorem}

\aliascntresetthe{Example}

\newaliascnt{Claim}{Theorem}

\aliascntresetthe{Claim}

\usepackage{chngcntr}
\counterwithin{Theorem}{section}

\numberwithin{equation}{section}

\allowdisplaybreaks

\newcommand{\Folner}{F\o{}lner}

\newcommand{\Oh}{{\rm O}}

\newcommand{\N}{\mathbb{N}}
\newcommand{\Z}{\mathbb{Z}}
\newcommand{\R}{\mathbb{R}}
\newcommand{\C}{\mathbb{C}}
\newcommand{\Q}{\mathbb{Q}}
\newcommand{\T}{\mathbb{T}}
\newcommand{\A}{\mathcal{A}}

\renewcommand{\epsilon}{\varepsilon}
\renewcommand{\leq}{\leqslant}
\renewcommand{\geq}{\geqslant}

\renewcommand{\subset}{\subseteq}

\newcommand{\F}{\mathcal{F}}

\newcommand{\E}{\mathbb{E}}
\newcommand{\B}{\mathcal{B}}
\newcommand{\1}{1}

\usepackage[normalem]{ulem}

\usepackage[makeroom]{cancel}

\DeclareMathSymbol{\shortminus}{\mathbin}{AMSa}{"39}
\newcommand{\e}{\varepsilon}
\newcommand{\norm}[1]{\left\Vert #1 \right\Vert}

\newcommand{\wt}{\widetilde}

\usepackage{pifont}
\newcommand{\plustimes}{\text{\ding{71}}}

\author{By~~{\scshape Florian~K.~Richter}~~and~~{\scshape Konstantinos Tsinas}}
\date{\small \today}
\title{\bfseries Density regularity of $\{x,x+y,xy\}$ in the integers}

\begin{document}

\maketitle
\begin{abstract}
\noindent Fix $k,s\in\N$.
We prove that there exists a subadditive density on $\N$ such that, for every polynomial $P\in\Z[y]$ of degree $k$ and with $P(0)=0$, every set of positive density contains configurations $\{x,x+P(y),xy^s\}$ for arbitrarily large $x>y\geq 2$.
This provides a density strengthening of Moreira's result on partition regularity for $\{x,x+y,xy\}$.
\end{abstract}

\tableofcontents
\thispagestyle{empty}



\section{Introduction}

In \cite{Hindman79a}, Hindman conjectured that every finite coloring of $\N=\{1,2,3,\ldots\}$ admits infinitely many pairs $x,y\in\N$ for which $\{x,y,x+y,xy\}$ is monochromatic.
This remains a central open problem in arithmetic Ramsey theory.
A first major breakthrough was obtained by Moreira \cite{Moreira17}, who proved that one can always find infinitely many $x,y\in\N$ such that $\{x,x+y,xy\}$ is monochromatic.
His proof combines topological recurrence with an inductive argument in the spirit of the topological proof of the polynomial van der Waerden theorem \cite{BL96}.
Moreira's method has since been refined and adapted to several other related sum--product problems; see \cite{BLM21, PRE21, Alweiss24, BS24, B25, B26, AG25, Alweiss26}.

A different approach to monochromatic sum--product patterns was developed in \cite{Richter26}, yielding an analytic proof that every finite coloring of $\N$ contains a monochromatic pair $\{x+y,xy\}$.
The key idea is to first prove a density theorem and then deduce the coloring result from it.
The proof of the density theorem relies crucially on a new type of subadditive density, which naturally gives rise to averages of mixed additive--multiplicative correlations.
These averages are then analyzed using methods from Fourier analysis and ergodic theory.

The purpose of the present paper is to develop this analytic framework further in order to treat three-point configurations such as $\{x,x+y,xy\}$.
More precisely, we prove a density theorem for configurations of the form
\begin{equation*}
\{x,x+P(y),xy^s\},
\end{equation*}
where $s\in\N$ and $P\in\Z[y]$ satisfies $P(0)=0$.
A central new ingredient is a localized version of the density introduced in \cite{Richter26}, and working with it requires new structural arguments from ergodic theory.

\subsection{Statement of main results}

Let $2^\N$ denote the power set of $\N$.
A \emph{subadditive density} on $\N$ is a function $D\colon 2^\N\to[0,1]$ satisfying $D(\emptyset)=0$ and $D(\N)=1$, which is monotone, in the sense that $D(A)\leq D(B)$ whenever $A\subseteq B$, and subadditive, in the sense that $D(A\cup B)\leq D(A)+D(B)$ for all $A,B\subseteq\N$.

Our main result is the following.

\begin{Theorem}\label{Thm_maindensity}
Fix $k,s\in\N$.
There exists a subadditive density $d^{\plustimes}$ on $\N$ such that, for every $E\subseteq\N$ with $d^{\plustimes}(E)>0$ and every $P\in\Z[y]$ of degree $k$ and satisfying $P(0)=0$, there exist $x>y\geq 2$ such that
\begin{equation*}
\{x,x+P(y),xy^s\}\subseteq E.
\end{equation*}
\end{Theorem}

We provide an explicit construction of the density $d^{\plustimes}$, but it requires, among other ingredients, the notion of an ergodic shift-invariant mean on $\N$.
We therefore postpone its precise construction to \cref{sec_density_construction}.

Subadditivity implies that every finite coloring of $\N$ has a color class of positive $d^{\plustimes}$-density.
Therefore, \cref{Thm_maindensity} immediately gives the following coloring result.

\begin{Corollary}\label{Cor_main}
For every finite coloring of $\N$, every $s\in\N$, and every $P\in\Z[y]$ satisfying $P(0)=0$, there exist infinitely many $x,y$ with $x>y\geq 2$ such that $\{x,x+P(y),xy^s\}$ is monochromatic.
\end{Corollary}

When $s=1$, \cref{Cor_main} recovers the polynomial form of Moreira's theorem \cite{Moreira17}; in particular, taking $P(y)=y$ yields the configuration $\{x,x+y,xy\}$.
For $s>1$, it recovers results first established in \cite[Proposition 4.1]{B26}.\footnote{An adaptation of Moreira's method in \cite{Moreira17} can actually establish partition regularity for $s\geq 2 $.}

Our proof of \cref{Thm_maindensity} entails studying averages of three-point mixed additive--multiplicative correlations.
To study these averages, we decompose the underlying correlations into structured and weak-mixing components.
The structured component supplies the required recurrence, while the contribution of the weak-mixing component is shown to be negligible. 
In particular, our proof provides new insight into how additive and multiplicative recurrence interact.
This approach is flexible, and we expect it to be useful beyond the configurations that we treat in this paper.

Our correlation-based approach also offers a possible route to quantitative results.
Green and Sawhney \cite{GS25} proved that, for all sufficiently large $r$, every $r$-coloring of $[N]=\{1,\ldots,N\}$ contains a monochromatic copy of $\{x+y,xy\}$ whenever
\begin{equation*}
N\geq \exp\exp(r^{50}).
\end{equation*}
Their argument builds on some of the ideas from the analytic proof of the partition regularity of $\{x+y,xy\}$ given in \cite{Richter26}.
It is therefore natural to ask whether the methods developed in the present paper can likewise be made effective, leading to quantitative bounds for configurations such as $\{x,x+y,xy\}$.

Finally, our analytic proof of the partition regularity of $\{x,x+y,xy\}$ may be relevant to Hindman's conjecture itself.
In \cite{GS16}, Green and Sanders proved the finite-field analogue of Hindman's conjecture, and 
their proof suggests a possible strategy for approaching the conjecture in the integers: first develop robust analytic control of the three-point correlations associated with $\{x,x+y,xy\}$, and then combine this with additional Ramsey-theoretic and structural input to force the full configuration $\{x,y,x+y,xy\}$.
The integer setting presents serious additional difficulties, most notably the non-amenability of the affine semigroup of the integers.
Nevertheless, by passing from the two-point configuration $\{x+y,xy\}$ to the genuine three-point configuration $\{x,x+y,xy\}$, the present paper moves towards the kind of analytic control that such a strategy would require.
Thus, the main theorem should be viewed both as a density strengthening of partition-regularity results involving sums and products and as a step towards a broader analytic theory of the interaction between addition and multiplication in the integers.

\subsection{Construction of density}
\label{sec_density_construction}

Let $\ell^{\infty}(\N)$ denote the space of all bounded complex-valued functions on $\N$, which forms a unital $C^{*}$-algebra under pointwise addition and multiplication.
We define the shift operator $\sigma\colon \ell^{\infty}(\N)\to \ell^{\infty}(\N)$ and the dilation operators $\tau_m\colon \ell^{\infty}(\N)\to \ell^{\infty}(\N)$, for $m\in \N$, via the formulas \begin{equation}\label{eq_definition of shift and dilation operators}
(\sigma f)(n)=f(n+1)\ \text{and } (\tau_m f)(n)=f(mn),\qquad n\in \N,~f\in\ell^\infty(\N). 
\end{equation} It is easy to check that these operators are bounded.

We need a notion of averaging that is insensitive to shifting, in the sense that the sequences $f$ and $\sigma f$ are assigned the same average. 

\begin{Definition}[Shift-invariant means]\label{D: invariant means}
A \emph{mean} on $\ell^{\infty}(\N)$ is a linear functional $\beta:\ell^\infty(\N)\to \C$, such that:
\begin{enumerate}
[label=(\roman{enumi}),ref=(\roman{enumi}),leftmargin=*]
\item $\beta$ is positive, i.e.,  $f\geq 0$ implies $\beta(f)\geq 0$;

\item $\beta$ is normalized, i.e., $\beta(\1_\N)=1$.
\end{enumerate}
A mean is called \emph{shift-invariant} if $\beta(\sigma f)=\beta(f)$ for all $f\in\ell^\infty(\N)$.
The collection of shift-invariant means is a convex and weak$^{*}$ compact set. We call the extreme points of this set \emph{ergodic (shift-invariant) means}. 
\end{Definition}

We also need to introduce the family of polynomial Bohr$_0$ sets, which provide the almost periodicity property that is crucial in our argument.
We also require a way to take a limit over this family. 
\begin{Definition}[Polynomial Bohr sets]
   Fix $k\in \N$. A polynomial Bohr$_{0}^{\leq k}$ set is a set of the form
    \begin{equation*}
        \left\{n\in \N\colon |\lambda_i^{P_i(n)}-1|\leq \e_i\  \text{for all } i\in [\ell]\right\}
    \end{equation*}for some  $\ell\in \N$, $\e_1,\ldots, \e_{\ell}>0$, $\lambda_1,\ldots, \lambda_\ell \in \mathbb{S}^1$ and non-constant polynomials $P_1,\ldots, P_\ell\in \Z[x]$ of degree at most $k$ and with $P_i(0)=0, \ 1\leq i \leq \ell$.
\end{Definition}

Note that the family of polynomial Bohr$_0^{\leq k}$ sets is a \emph{$\pi$-system} on $\N$, which means that it is closed under finite intersections. We denote this $\pi$-system by $\text{PB}_0^{\leq k}$ for brevity. 

\begin{Definition}[Limit superior along $\pi$-systems]\label{D: limsup for filters}
Suppose $\mathcal{F}$ is a $\pi$-system on $\N$. Given a bounded function $f\colon \mathcal{F}\to\R$, we define
\[
\limsup_{A\in \mathcal{F}} f(A)= \inf_{B\in\mathcal{F}}\ \sup_{C\in \F} f(C\cap B).
\]
\end{Definition}

When averaging relative to a subset of $\N$, we will use the following normalization.
If $B\subseteq\N$ is a set and $\gamma$ is a mean on $\ell^\infty(\N)$ with $\gamma(\1_B)\neq 0$, we define
\begin{equation}\label{eq:relativized-mean}
    \gamma_B(f):=\frac{\gamma(f\cdot \1_B)}{\gamma(\1_B)}.
\end{equation}
Thus $\gamma_B$ should be thought of as the mean $\gamma$ conditioned on the set $B$.

\begin{Remark}
We will use the notation $\gamma_B$ only in the case when $\gamma$ is a shift-invariant mean and $B$ is a polynomial Bohr$_{0}^{ \leq k}$ set, and we claim that under these assumptions $\gamma_B$ is always well-defined.

Indeed, if $B$ is such a set, then $\1_B$ is \emph{almost convergent}: every shift-invariant mean assigns to $\1_B$ the same value, namely the natural density $d(B)$ of $B$.
By \cite{Lorentz48}, this is equivalent to the uniform Cesàro convergence
\begin{equation*}
\lim_{N-M\to\infty}\frac{1}{N-M}\sum_{M\leq n<N}\1_B(n)=d(B).
\end{equation*}
For polynomial Bohr$_{0,\leq k}$ sets, this property follows from a classical argument of Furstenberg: one realizes the sequence $\1_B(n)$ as the values of a Riemann-integrable function along the orbit of a point in a uniquely ergodic system; see, for instance, \cite[Section 4.4]{EW11}.
We leave the details to the interested reader.
It follows that $\gamma(\1_B)=d(B)>0$ for every shift-invariant mean $\gamma$, and hence $\gamma_B$ is well defined.
\end{Remark}

We now define the density used in our main theorem.

\begin{Definition}[The density $d^{\plustimes}$]
\label{def_new_density}
Fix $k,s\in\N$.
Let $\beta$ and $\gamma$ be shift-invariant means on $\ell^\infty(\N)$, and assume that $\beta$ is ergodic.
For $E\subseteq\N$, define
\begin{equation}\label{eq:definition-density}
d^{\plustimes}(E):=\limsup_{B\in \mathrm{PB}_0^{\leq k}}\gamma_B\big(m\mapsto \beta(n\mapsto \1_E(m^s n))\big).
\end{equation}
\end{Definition}

Note that $d^{\plustimes}$ satisfies the axioms of a subadditive density:
\begin{enumerate}[label=(\roman{enumi}),ref=(\roman{enumi}),leftmargin=*]
\item $d^{\plustimes}(\emptyset)=0$ and $d^{\plustimes}(\N)=1$;
\item if $A\subseteq B$, then $d^{\plustimes}(A)\leq d^{\plustimes}(B)$;
\item $d^{\plustimes}(A\cup B)\leq d^{\plustimes}(A)+d^{\plustimes}(B)$ for all $A,B\subseteq\N$.
\end{enumerate}
\medskip

The definition of our density in \eqref{eq:definition-density} might seem complicated at first reading, but one can think of it as an invariant-mean version of the double average
\begin{equation}\label{eq:double-average-heuristic}
\limsup_{B\in \mathrm{PB}_0^{\leq k}}
\lim_{M\to\infty}\lim_{N\to\infty}
\E_{m\in B\cap[M]}\E_{n\in[N]}\1_E(m^s n).
\end{equation}
We use \eqref{eq:definition-density} instead of \eqref{eq:double-average-heuristic} for two reasons.
First, the limits in \eqref{eq:double-average-heuristic} need not exist for an arbitrary set $E$, whereas invariant means provide a stable way to take such averages without worrying about the existence of limits.
Second, the inner average $\E_{n\in[N]}$ need not have the recurrence properties required for our argument; choosing $\beta$ to be ergodic ensures that the inner average has the ergodic recurrence behavior needed in our proof of \cref{Thm_maindensity}.

This definition should be compared with the density used in \cite{Richter26}, where no ergodicity assumption on the inner average was needed, and where the outer dilation was taken over the simpler family $\{a\N:a\in\N\}$ rather than over the family $\mathrm{PB}_0^{\leq k}$.  This latter change reflects the increased complexity of the three-point configuration: controlling the mixed sum-product expressions requires decomposition results from ergodic theory that genuinely correspond to a higher complexity factor than the one used in the decomposition in \cite{Richter26}.


\subsection*{Structure of the paper}

The rest of the paper is organized as follows.

In Section~2, we review the necessary background from ergodic theory, including basic notions concerning measure-preserving systems and results on ergodic averages along polynomial iterates.
We also collect the required facts about invariant means and \Folner{} sequences on $\N$.

In Section~3, we state our main dynamical result (Theorem~\ref{prop_dynamical statement}) and use a variant of Furstenberg's correspondence principle to deduce the main combinatorial result (Theorem~\ref{Thm_maindensity}) from it.

Finally, Section~4 is devoted to the proof of the dynamical theorem.
The proof utilizes a standard ergodic-theoretic splitting into structured and pseudorandom components and combines recurrence arguments with results from Ramsey theory.

\subsection*{Notation}
We denote by $\N, \Z,\Q,\R,\C$ the sets of natural, integer, rational, real and complex numbers, respectively. We also define $\mathbb{S}^1=\{z\in \C\colon |z|=1\}$ and the torus $\T=\R/\Z$. Finally, we also let $e(t):=e^{2\pi it}$ for $t\in \R$.

We set $[N]=\{1,2,\ldots, N\}$.
We denote the average of a sequence $f(n)$ along a finite, non-empty set $F\subset \N$ by \begin{equation*}
    \E_{n\in F}f(n)=\frac{1}{|F|}\sum_{n\in F}f(n).
\end{equation*}

We will use standard asymptotic notation throughout this paper. If $f,g$ are sequences, we write $f=\Oh(g)$ if there exists a constant $C$ such that $|f(n)|\leq C|g(n)|$.

We will use lowercase letters like $f,g:\N\to \C$ for sequences in $\ell^{\infty}(\N)$, and we use capital letters $F, G$ to denote functions on metric spaces or probability spaces. We reserve lowercase Greek letters like $\beta,\gamma$ for linear functionals. 

If $E$ is a subset of $\N$, we let $|E|$ be its cardinality and for a fixed $h\in \N$, we also set $E-h=\{n\in \N\colon n+h\in E\}$ and $E/h=\{n\in \N\colon hn\in E\}$.

\subsection*{Acknowledgments}
This work was supported by the Swiss National Science Foundation grant TMSGI2-211214. We thank Joel Moreira for helpful discussions during the preparation of this paper. We also thank Matt Bowen for bringing the partition-regularity results of \cite{B26} to our attention.

\subsection*{AI disclosure statement}
The authors used the AI reasoning model \emph{GPT-5.5 Thinking}
to generate an initial draft of the proof of
Lemma~\ref{lem_ergodicity passes to factors}, a fairly standard
result asserting that the restriction of an ergodic mean to an
invariant subalgebra remains ergodic. The proof was then carefully checked and revised.
Furthermore, the AI model \emph{GPT-6 Astra} identified  a missing case in \cref{prop_optimal factor for 2 polynomials} in a previous version of the article, which led to a logical gap when applied on Section~4. 
The error was subsequently corrected by the authors.
All other ideas and proofs are due to the authors. Apart from the uses described above, AI tools were used only to improve the wording and readability of the manuscript. The authors assume responsibility for all content.

\section{Preliminaries}\label{Section_background}

In this section, we collect some background information from ergodic theory and functional analysis that will be relevant in the proof of Theorem \ref{Thm_maindensity}.

\subsection{Basics of measure-preserving systems}

A measurable map $T$ on a probability space $(X,\mathcal{X},\mu)$ is called \emph{measure-preserving}, if $\mu(T^{-1}A)=\mu(A)$ for all $A\in \mathcal{X}$. We call $(X,\mathcal{X},\mu,T)$ a \emph{measure-preserving system}.
A measure-preserving system is \emph{ergodic} if every invariant set $A\in \mathcal{X}$ (that is, any set $A\in \mathcal{X}$ satisfying $T^{-1}A=A$) has measure $0$ or $1$. It is called \emph{invertible} if $T^{-1}$ exists almost everywhere and is measurable.

We say that a function $F\in L^2(\mu)$ with $F\neq 0$ is an \emph{eigenfunction} of $(X,\mathcal{X},\mu,T)$, if there exists $\lambda\in \C$, such that $F(Tx)=\lambda F(x)$ for $\mu$-almost every $x$. In this context, $\lambda$ is called the \emph{eigenvalue} associated to the eigenfunction $F$. 

We say that a measure-preserving system $(Y,\mathcal{Y},\nu,S)$ is a \emph{factor} of another measure-preserving system $(X,\mathcal{X},\mu,T)$, if there exist $X'\subset X$, $Y'\subset Y$ of full measure that are invariant under $T$ and $S$ respectively and a map $p:X'\to Y'$ such that $\nu=\mu\circ p^{-1}$ and $p\circ T(x)=S\circ p(x)$ for all $x\in X'$. The map $p$ is called a \emph{factor map}.
We say that $(X,\mathcal{X},\mu,T)$ and $(Y,\mathcal{Y},\nu,S)$ are \emph{isomorphic}, if there exists a factor map that is invertible modulo null sets.

A factor of $(X,\mathcal{X},\mu,T)$ can always be realized as a $T$-invariant sub-$\sigma$-algebra of $\mathcal{X}$. To be more precise, there is a sub-$\sigma$-algebra $\mathcal{Y}$ of $\mathcal{X}$ with $T^{-1}\mathcal{Y}=\mathcal{Y}$ (modulo null sets) and such that the factor is isomorphic to the measure-preserving system $(X,\mathcal{Y},\wt{\mu},T)$, where $\wt{\mu}=\mu|_{\mathcal{Y}}$ and $T$ is regarded as a measurable map on $(X,\mathcal{Y},\wt{\mu})$.

The sub-$\sigma$-algebra generated by rational eigenfunctions (meaning functions that satisfy $T^qF=F$ for some $q\in \N$)
is called the \emph{rational Kronecker factor} of $(X,\mathcal{X},\mu,T)$. We will denote this factor by $\mathcal{K}_{\mathrm{rat},T}$.
Similarly, the sub-$\sigma$-algebra spanned by the eigenfunctions in an ergodic system is called the \emph{Kronecker factor}. It is the largest factor of $(X,\mathcal{X},\mu, T)$ that is isomorphic to a rotation on a compact abelian group. We will denote this factor by $\mathcal{K}_{T}$.

A function $F$ on $L^2(\mu)$ is called \emph{weak-mixing}, if \begin{equation*}
   \lim_{k\to+\infty} \E_{m\in \Phi_k} \Big|\int \overline{F}\cdot T^mF\, d\mu \Big|=0
\end{equation*}for all \Folner\ sequences $(\Phi_k)_{k\in \N}$ (see Definition \ref{D: definition of Folner} below). 

The following theorem asserts that in any measure-preserving system, we can split any function into a component that is measurable with respect to the Kronecker factor and a component that is weak-mixing.

\begin{Theorem}[{cf.\ \cite[Theorem~2.3.4]{Krengel85} or~\cite[Chapter 4,~Proposition~19]{HK18}}]\label{Thm_compact+weak mixing}
    Let $(X,\mathcal{X},\mu,T)$ be a measure-preserving system and define \begin{equation*}
        \mathcal{H}_\text{c}=\overline{\operatorname{span}\{F\in L^2(\mu)\colon F \ \text{is an eigenfunction}\} }
    \end{equation*}and 
    \begin{equation*}
        \mathcal{H}_{\text{wm}}=\{F\in L^2(\mu)\colon F \ \text{is a weak-mixing function}\}.
    \end{equation*}Then, $L^2(\mu)=\mathcal{H}_\text{c}\oplus \mathcal{H}_{wm}$.
\end{Theorem}This implies the following corollary.

\begin{Corollary}\label{cor_decomposition via Kronecker}
    Let $(X,\mathcal{X},\mu,T)$ be a measure-preserving system. Then, any function in $F\in L^2(\mu)$ admits a decomposition \begin{equation*}
        F=F_{\mathrm{str}}+F_{\mathrm{unf}},
    \end{equation*}where $F_{\mathrm{str}}\in \mathcal{H}_\mathrm{c}$ and $F_{\mathrm{unf}}\in \mathcal{H}_{\mathrm{wm}}$. Furthermore, if $F$ takes values in an interval $[a,b]$, we can choose $F_{\mathrm{str}}$ to take values in $[a,b]$ as well.
\end{Corollary}

\subsection{Characteristic factors for some polynomial families}

Next, we prove that the Kronecker factor is characteristic for ergodic averages involving two polynomial iterates and a polynomial phase twist.
This is a key step in our analysis of sum-product patterns, since it lets us apply the compact/weak mixing decomposition from \cref{Thm_compact+weak mixing}.
We prove this using established facts from \cite{Ku23} and \cite{FK25} about characteristic factors for polynomial ergodic averages.

We call a finite collection of polynomials \emph{essentially distinct} if the difference of any two distinct polynomials in this collection yields a non-constant polynomial.

\begin{Proposition}\label{prop_optimal factor for 2 polynomials}Let $P_1, P_2$ be polynomials with integer coefficients and zero constant terms, and let $Q$ be a polynomial with real coefficients. Assume also that $P_1, P_2$ have the same degree and are essentially distinct and that the degree of $Q$ is smaller than or equal to the degree of $P_1$.
Then, for any invertible ergodic measure-preserving system $(X,\mathcal{X},\mu,T)$, and any functions $F,G\in L^{\infty}(\mu)$ such that at least one of $F,G$ is weak-mixing, we have \begin{equation*}
\lim_{M\to\infty}\Big\|  \E_{m\in \Phi_M} e( Q(m)) T^{P_1(m)}F\cdot T^{P_2(m)}G \Big\|_{L^2(\mu)}=0
    \end{equation*}
for any \Folner\ sequence $(\Phi_M)_{M\in \N}$.
\end{Proposition}

\begin{proof}
Throughout the proof, we may assume $F$ is a weak-mixing function and that $Q(0)=0$. We abuse terminology and say a function $F\in L^{\infty}(\mu)$ is orthogonal to a factor $(Y, \mathcal{Y}, \nu,S)$, if the conditional expectation\footnote{Recalling that a factor corresponds to an invariant sub-$\sigma$-algebra, the conditional expectation with respect to a factor is defined to be equal to the conditional expectation with respect to this $\sigma$-algebra.} $\E_{\mu}(F|\mathcal{Y})$ is equal to zero. 
We consider two cases depending on whether the polynomials $P_1$ and $P_2$ are dependent or not. We fix $k$ as the degree of the polynomials $P_1,P_2$

Assume first that $P_1, P_2$ are linearly independent. Note that the real vector space $\R_k[x]$ of polynomials of degree at most $k$ and with zero constant term has dimension $k$. Thus, we can find a collection $P_3,\ldots, P_{k}$ of polynomials with integer coefficients, such that $\{P_1,\ldots, P_{k}\}$ spans
$\R_k[x]$.

Write 
\[
Q(m)=\sum_{i=1}^k \theta_i P_i(m)
\]where $\theta_i\in \R$. For each $\theta_i$ we consider a rotation $(\T,\mu_i,S_i)$ by $\theta_i$.

 We rewrite our averages as 
 \[
\left\| \E_{m\in \Phi_M} (T\times S_1)^{P_1(m)} F'\cdot (T\times S_2)^{P_2(m)} G'    \prod_{i=3}^k (T\times S_i)^{P_i(m)}H_i \right\|_{L^2(\nu)}
 \]
where $F'(x,y)=F(x)e(y)$, $G'(x,y)=G(x)e(y)$ and $H_i(x,y)=e(y)$ for every $3\leq i\leq k$ and every $(x,y)\in X\times\T$.

The transformations $T\times S_i$ are commuting and the polynomials $P_1,\ldots, P_k$ are linearly independent. Therefore, \cite[Theorem 2.10]{FK25} implies that if $F'$ is orthogonal to the rational Kronecker factor $\mathcal{K}_{\mathrm{rat}, T\times S_1}$, then the previous average is zero.

We justify the orthogonality by showing that for every $r\in \N$, we have that the averages
\begin{equation*}
    \E_{m\in \Phi_M} (T\times S_1)^{rm}F'
\end{equation*}
converge to $0$ in $L^2(\nu)$. Plugging in the formula for $F$, this reduces to proving that the averages
\begin{equation*}
    \E_{m\in \Phi_M} e(rm\theta_1)T^{rm}F
\end{equation*}
converge to zero in $L^2(\mu)$. But this follows from the fact that $F$ is a weak-mixing function and, thus, its spectral measure has no point masses.

We are left with the case that $P_1,P_2$ are dependent. Then, we can find a polynomial $P\in \Z_k[x]$ such that $P_1=r_1P$ and $P_2=r_2P$ for some $r_1,r_2\in \Z$.
Since $P$ has degree $k$, the set $\{P,n,n^2,\ldots, n^{k-1}\}$
forms a basis for $\R_k[x]$ and we can write 
\begin{equation*}
    Q(m)=\theta P(m) +\sum_{i=1}^{k-1}\theta_i m^i.
\end{equation*}
Consider the space $X\times \T^k$ with the transformation
\[
S(x,y)=\Big(Tx,y+\big(\frac{\theta}{r_1},\theta_1,\ldots, \theta_{k-1}  \big)\Big), \qquad x\in X, y\in \T^k
\]and the measure $\mu\times \lambda_{\T^k}$, where $\lambda_{\T^k}$ stands for the Lebesgue measure on $\T^k$. Since $T$ is ergodic with respect to $\mu$, almost every ergodic component of $\mu\times \lambda_{\T^k}$ projects to $\mu$ under the corresponding coordinate map $X\times \T^{k}\to X$. We fix $\nu$ to be such an ergodic measure on $X\times \T^k$.

We rewrite our averages as 
\[
\Big\|\E_{m\in \Phi_M} S^{r_1P(m)}F'\cdot S^{r_2P(m)}G'
\prod_{j=1}^{k-1}S^{m^j}H_j
\Big\|_{L^2(\nu)}
\]
where $F'(x,y_1,\ldots, y_k)=F(x)e(y_1)$, $G'(x,y_1,\ldots, y_k)=G(x)$ and $H_j(x,y_1,\ldots, y_k)=e(y_{j+1})$ for all $(x,y_1,\ldots, y_k)\in X\times \T^k$.

We claim that if $F'$ is orthogonal to the Kronecker factor, then the previous average is zero. Assuming this claim for now, we conclude the proof, since $F'$ is a weak-mixing function
on the product system. This follows from the identity 
\[
\left|\int \overline{F'}\cdot S^n F'\,d\nu\right|=\left|\int \overline{F}\cdot T^n F \, d\mu\right|
\]and the hypothesis that $F$ is a weak-mixing function.

In order to establish that the Kronecker factor is characteristic for our averages, it suffices by \cite[Corollary 1.14]{Ku23}
to show that the family of polynomials 
$$
\{r_1P(m),r_2P(m), m,\ldots, m^{k-1}\}
$$
which appears in the iterates has algebraic complexity at most 1. This means that if there are polynomials $Q_0,Q_1,\ldots, Q_{k+1}$
with real coefficients such that \begin{equation}\label{eq_polynomial relations}
    Q_0(x)+Q_1(x+r_1P(y))+Q_2(x+r_2P(y))+\dots +Q_{k+1}(x+y^{k-1})=0,
\end{equation}then the polynomials $Q_0,\ldots, Q_{k+1}$ have degree at most $1$.

Assuming that there exist polynomials satisfying \eqref{eq_polynomial relations} with degrees at least 2, then after differentiating sufficiently many times, we find polynomials $\wt{Q}_i$ of degree at most 2 for which \eqref{eq_polynomial relations} is satisfied. 
Therefore, it suffices to prove our claim when the polynomials $Q_i$ have degree at most 2.

Write $Q_i(x)=a_ix^2+b_ix+c_i$ with $a_i,b_i,c_i\in \R$ and assume that not all $a_i$ are zero.
Then, \begin{equation*}
    a_0x^2+a_1(x+r_1P(y))^2+a_2(x+r_2P(y))^2+\sum_{i=3}^{k+1} a_i(x+y^{i-2})^2 +\Big(\sum_{i=0}^{k+1}b_i\Big)x+R(y)=0,
\end{equation*}where $R(y)$ is a polynomial of degree at most $k$.

Comparing the coefficient of the term $x$, we get \begin{equation*}
    2(a_1r_1+a_2r_2)P(y) +2\sum_{i=3}^{k+1}a_iy^{i-2}+\sum_{i=0}^{k+1}b_i=0
\end{equation*}The polynomial $P(y)$ has degree $k$. Therefore, we get $a_1r_1+a_2r_2=a_3=\ldots=a_{k+1}=0$. 

Now we compute the coefficient of $y^{2k}$, which appears only in the terms with $a_1,a_2$, to infer that
\[
a_1r_1^2+a_2r_2^2=0.
\]
We conclude that $a_1=a_2=0$. Finally, the coefficient of $x^2$ yields $a_0=0$. 

To summarize, we have that $a_0=a_1=\ldots=a_{k+1}=0$, a contradiction. Thus, the polynomial family has algebraic complexity at most 1 and the conclusion follows.
\end{proof}

\subsection{Means on $C^*$-algebras and \Folner\ sequences}

Here, we gather some basic lemmas about means on subalgebras of $\ell^{\infty}(\N)$. 

Recall the definitions of the shift and dilation operators $\sigma$ and $\tau_m$. 
A $C^{*}$-subalgebra $\A$ of $\ell^{\infty}(\N)$ is called \emph{shift-invariant} if for every element $f\in\A $ we have $\sigma f\in \A$. It is called \emph{affinely invariant}, if additionally $\tau_{m}f\in \A$ for every $m\in \N$. Given an at most countable collection of elements $(f_i)_{i\in \N}$ in $ \ell^{\infty}(\N)$, we can construct the smallest unital, affinely invariant  $C^{*}$-subalgebra of $\ell^\infty(\N)$ that contains all functions $f_i$. We will call this the \emph{affinely invariant $C^{*}$-algebra generated by} $(f_i)_{i\in \N}$. Note that any such algebra is separable and contains the sequence $\1_{\N}$, which will always be the unit.

We can define shift-invariant means on $C^*$-algebras using the same properties as we did for $\ell^{\infty}(\N)$. 
\begin{Definition}[Means on $C^*$-algebras]\label{def_means on algebras}
    A \emph{mean} on a unital $C^*$-subalgebra $\A$ of $\ell^{\infty}(\N)$ is a linear functional $\beta:\A\to \C$, such that:
\begin{enumerate}
[label=(\roman{enumi}),ref=(\roman{enumi}),leftmargin=*]
\item $\beta$ is positive, i.e., $\beta(f)\geq 0$ whenever $f$ is a positive\footnote{An element in a $C^{*}$-algebra $\mathcal{A}$ is called positive if it takes the form $\sum_{i=1}^k b_i\cdot b_i^{*}$, for some $b_1,\ldots, b_k\in \mathcal{A}$.} element,

\item $\beta$ is normalized, i.e., $\beta(\1_\N)=1$.
\end{enumerate}

A mean on a shift-invariant $C^*$-subalgebra $\A$ is called \emph{shift-invariant} if $\beta(\sigma f)=\beta(f)$ for all $f\in\A$.
The collection of shift-invariant means is a convex and weak$^{*}$ compact set. We call the extreme points of this set the \emph{ergodic (shift-invariant) means} of $\A$. 
\end{Definition}

A simple way to create shift-invariant means on a shift-invariant $C^*$-subalgebra $\A$ is through \Folner\ sequences.

\begin{Definition}[\Folner\ sequences]\label{D: definition of Folner}
    A sequence $(\Phi_k)_{k\in \N}$ of finite non-empty subsets $\Phi_k\subseteq \N$ is called a \emph{\Folner\ sequence} if it is asymptotically shift invariant in the sense that \begin{equation*}
        \lim_{k\to+\infty} \frac{|\Phi_k\cap (\Phi_k-h)|}{|\Phi_k|}=1 \ \text{for all } h\in \N.
    \end{equation*}
\end{Definition}

An analogous definition of a \Folner\ net can be given by replacing the index set $\N$ with another directed set.

The most classical result of a \Folner\ sequence is a sequence of intervals with lengths tending to infinity.

\Folner\ sequences provide a natural way to construct means on subalgebras of $\ell^{\infty}(\N)$. Indeed, suppose that $\A$ is a unital, separable, shift-invariant $C^*$-subalgebra of $\ell^{\infty}(\N)$ and let ${\bf \Phi} = (\Phi_k)_{k\in \N}$ be a \Folner\ sequence. 
For any finite set $F$, we define the mean \begin{equation*}
    A_F(f):=\frac{1}{|F|}\sum_{n\in F} f(n).
\end{equation*}
Note that a sequence (or more generally a net) of the means $A_{\Phi_k}$ defines a sequence (net) of linear functionals on $\ell^{\infty}(\N)$ of norm 1. By the Banach-Alaoglu theorem, we can always find a subsequence (subnet) that converges in the weak$^*$ topology to some functional $A_{{\bf \Phi}}$. It is not difficult to show
that $A_{{\bf \Phi}}$ defines a shift-invariant mean on $\A$.
We say that a shift-invariant mean $\beta$ is a limit along a \Folner\ net, if there exists a \Folner\ net $(\Phi_i)_{i\in I}$, such that the functionals $A_{\Phi_i}$ converge to $\beta$ in the weak$^*$-topology. We say that the mean is a limit along a \Folner\ sequence, if we can take $I=\N$ in the previous definition.

The next lemma implies that all means on a separable $C^*$-subalgebra of $\ell^{\infty}(\N)$ take this form.

\begin{Lemma}\label{L: mean on separable is Folner}
    Let $\A$ be a unital, separable, shift-invariant $C^*$-subalgebra of $\ell^{\infty}(\N)$
    and let $\beta$ be a shift-invariant mean on $\A$. Then, there exists a \Folner\ sequence ${\bf \Phi} = (\Phi_k)_{k\in \N}$, such that $\beta(f)=A_{{\bf \Phi}}(f)$ for every $f\in \A$.
\end{Lemma}

\begin{proof}
Our result follows from \cite[Corollary 4.6]{HS9}, which asserts that every shift-invariant mean on $\ell^{\infty}(\N)$ is a limit along a \Folner\ net. This implies that any shift-invariant mean on $\A$ is also a limit along \Folner\ nets since it is the restriction of a shift-invariant mean of $\ell^{\infty}(\N)$ to $\A$ (by the invariant Hahn--Banach theorem).
Since $\A$ is a separable space, the subspace topology inherited from the weak*-topology on the unit ball of the dual of $\ell^{\infty}(\N)$ is metrizable. Thus, the unit ball of the dual of $\A$ is a compact metric space.
Using the first countability, it is not then difficult to show that for any cluster point of a \Folner\ net $(\Phi_i)_{i\in I}$, we can construct a ``subsequence'' $(\Phi_{i_n})_{n\in \N}$ with indices $i_n\in I$ satisfying $i_n\leq i_{n+1}$ for every $n\in \N$, which also converges to that point. 
We infer that any shift-invariant mean on $\A$ is also a limit along a \Folner\ sequence. 
\end{proof}

The following lemma implies that the ergodicity of a mean is preserved when passing to shift-invariant subalgebras.

\begin{Lemma}\label{lem_ergodicity passes to factors}
    Let $\beta$ be an ergodic shift-invariant mean on $\ell^{\infty}(\N)$. Let $\A$ be any unital, shift-invariant $C^*$-subalgebra of $\ell^{\infty}(\N)$. Then, the restriction $\beta|_{\A}$ is an ergodic mean on $\A$.
\end{Lemma}

\begin{proof}
Suppose that $\beta|_{\A}$ is not an extreme point among the shift-invariant means on $\A$.
Then, there exist shift-invariant means $\beta_1,\beta_2$ on $\A$ and a number
$0<t<1$ such that
$$\beta|_\A=t\beta_1+(1-t)\beta_2.$$
Define
$
\gamma_1:=t\beta_1,
\gamma_2:=(1-t)\beta_2.
$
Then
$$
\gamma_1+\gamma_2=\beta|_\A,
$$
and therefore
$$
0\leq \gamma_1(f)\leq \beta|_\A(f)
\ \text{for all } f\geq 0,
$$since $\gamma_2$ is a positive functional.

By the ordered Hahn--Banach extension theorem\footnote{Since the functionals take complex values, while the ordered Hahn--Banach theorem is stated in most sources for real valued functionals, one may proceed as follows: on the (real) vector space of self-adjoint elements of $\A$, we note that $\gamma_1$ is dominated by the sub-linear functional $p(f)=\inf\{\beta(g), g\geq 0, g\geq f\}$. We then apply the Hahn--Banach theorem to construct a real functional on the real vector space of self-adjoint elements of $\ell^{\infty}(\N)$, which is dominated by $\beta$ on positive elements, and then complexify it via routine procedure.} (\cite[Theorem 3.2]{Rudin1991}), there exists
a positive functional $\widetilde{\gamma}_1$ on $\ell^\infty(\N)$ such that
\[
\widetilde{\gamma}_1|_\A=\gamma_1
\]
and
\[
 \widetilde{\gamma}_1(f)\leq \beta(f) 
\]for all positive $f\in \ell^{\infty}(\N)$.
Note that the last inequality forces the functional $\wt{\gamma_1}$ on the $C^*$-algebra $\ell^{\infty}(\N)$ to be bounded and combined with the fact that $\norm{\wt{\gamma}_1}=\norm{\gamma_1}=\gamma_1(1)=\wt{\gamma}(1)$, this implies the positivity of $\wt{\gamma}_1$. 
However, it is not necessarily shift-invariant. To overcome this problem, we take a shift-invariant mean
$\gamma$ on $\ell^\infty(\N)$ and average along shifts. More specifically, for
$f\in\ell^\infty(\N)$, define
$$
\gamma_1^*(f)
:=
\gamma\big(
n\mapsto \widetilde{\gamma}_1(\sigma^n f)
\big).
$$
Then $\gamma_1^*$ is a positive shift-invariant functional on $\ell^\infty(\N)$.
Moreover, since
\[
0\leq \widetilde{\gamma}_1(\sigma^n f)
\leq \beta(\sigma^n f)
=\beta(f)
\qquad \text{for }f\geq 0,
\]
we have
\[
0\leq \gamma_1^*\leq \beta .
\]
Since $\gamma_1$ is already shift-invariant on $\A$, for every $f\in \A$,
\[
\gamma_1^*(f)
=
\gamma \big(
n\mapsto \widetilde{\gamma}_1(\sigma^n f)
\big)
=
\gamma\big(
n\mapsto \gamma_1(\sigma^n f)
\big)
=
\gamma\!\left(
n\mapsto \gamma_1(f)
\right)
=
\gamma_1(f).
\]
Hence
\[
\gamma_1^*|_\A=\gamma_1 .
\]
Set $\gamma_2^*:=\beta-\gamma_1^*$.
Since $0\leq \gamma_1^*\leq \beta$, the functional $\gamma_2^*$ is positive and
shift-invariant.

We have established that the functionals $\gamma_1^*, \gamma_2^*$ are positive and shift-invariant, so we normalize them to make them means.
Note that
\[
\gamma_1^*(1)=\gamma_1(1)=t\in (0,1).
\]
Consequently,
\[
\beta
=
t\frac{\gamma_1^*}{t}
+
(1-t)\frac{\gamma_2^*}{1-t},
\]
where both $\gamma_1^*/t$ and $\gamma_2^*/(1-t)$ are shift-invariant means on
$\ell^\infty(\N)$.
Since $\beta$ is extreme among shift-invariant means on $\ell^\infty(\N)$, it
follows that
\[
\frac{\gamma_1^*}{t}=\beta.
\]
Restricting to $\A$ gives
\[
t\beta_1
=
\gamma_1
=
\gamma_1^*|_\A
=
t\,\beta|_\A ,
\]
which implies that
\[
\beta_1=\beta|_\A .
\]
The same argument applied to $\gamma_2$ yields
\[
\beta_2=\beta|_\A .
\]

In conclusion, every convex decomposition of $\beta|_A$ into shift-invariant means on
$\A$ is trivial. Hence $\beta|_A$ is an extreme point of the set of
shift-invariant means on $\A$.
\end{proof}

\section{Dynamical reformulation of the main result}

The purpose of this section is to reduce our main combinatorial result, Theorem~\ref{Thm_maindensity}, to a dynamical result about recurrence in ergodic systems.

The passage to a dynamical statement relies on the following version of Furstenberg's correspondence principle.
Similar results appear numerous times in the literature, though they are typically stated with averages along \Folner\ sequences in place of shift-invariant means.

\begin{Proposition}[Correspondence principle]\label{prop_correspondence principle}
     Let $\A$ be a separable, unital, shift invariant $C^{*}$-subalgebra of $\ell^{\infty}(\N)$ and let $\beta$ be an ergodic shift-invariant mean on $\A$. Then there exists a compact metric space $X$, a continuous map $T\colon X\to X$, a $T$-invariant Borel probability measure $\mu$ on $X$ that is ergodic with respect to $T$, and an isometric $C^{*}$-isomorphism  $\Psi:\A\to C(X)$, such that for all positive integers $\ell$ and sequences $f_1,\ldots,f_{\ell} \in \A$, there exist $F_{1},\ldots, F_\ell\in C(X)$, such that for all non-negative $h_1\ldots, h_{\ell}\in \Z$, we have
        \begin{equation}\label{eq_furstenberg system identity0}
    \beta\bigg(n\mapsto f_1(n+h_1)\cdot\ldots\cdot f_{\ell}(n+h_{\ell})\bigg)=\int T^{h_1}F_1\cdot \ldots\cdot T^{h_{\ell}}F_{\ell}\, d\mu.
    \end{equation}
   In addition, if $f_1,\ldots, f_k$ are positive elements of $\A$, then the functions $F_1,\ldots, F_k$ are non-negative.
\end{Proposition}

\begin{proof}Recall that $\sigma$ denotes the shift operator on $\ell^{\infty}(\N)$.
    We let $$X=\{\phi:\A\to \C\colon \ \phi \text{ is a multiplicative linear functional}\}$$
be the Gelfand spectrum of $\A$. This is a compact Hausdorff space with the weak*-topology (the subspace topology inherited by the weak* topology on the dual of $\A$). The weak*-topology on the unit ball of $\A^{*}$ is metrizable since $\A$ is separable. Thus, $X$ is a compact metrizable space. Finally, by the Gelfand-Naimark theorem, there exists an isometric $C^*$-isomorphism $\Psi: \A\to C(X)$.

We denote by $T$ the transformation $$T: X\to X, \ T\phi(a)=\phi(\sigma(a))\ \text{for every } \phi\in X,$$ 
which is easily checked to be a well-defined, continuous function on $X$. 

To construct the measure on $X$, we consider the push-forward of the linear functional $\beta$. Namely, if we define $\rho=\beta\circ \Psi^{-1} $, then it is straightforward to check that this is a positive linear functional on the vector space $C(X)$ (with the supremum norm). Invoking the Riesz-Markov-Kakutani representation theorem, we can find a Radon measure $\mu$ on $X$ characterized by \begin{equation}\label{eq_characterization of mu}
    \rho(F)= \int F \, d\mu\ \text{for every } F\in C(X).
\end{equation} Since $\beta$ is shift-invariant, it follows that the measure $\mu$
is $T$-invariant.

We claim that the measure $\mu$ is ergodic.  It is known (see \cite[Theorem 4.1, Theorem 4.4]{EW11}) that the space of Borel $T$-invariant probability measures on the compact metric space $X$ is a non-empty, convex set and its extreme points are precisely the ergodic $T$-invariant probability measures on $X$. Arguing by contradiction, we suppose that $\mu$ is not ergodic and, thus, not an extreme point for the set of $T$-invariant probability measures on $X$. This implies that we can find $T$-invariant measures $\mu_1,\mu_2$ and $t\in (0,1)$ such that
$$\mu=t\mu_1+(1-t)\mu_2.$$
Naturally, it follows that for the functionals $\rho_1,\rho_2$ corresponding to $\mu_1,\mu_2$, we also have $$\rho=t\rho_1+(1-t)\rho_2.$$
Consequently, $$\beta=t\beta_1+(1-t)\beta_2,$$where  $\beta_1=\rho_1\circ \Psi $ and $\beta_2=\rho_2\circ \Psi$ are shift-invariant means on $\A$. This contradicts the ergodicity assumption on $\beta$, and the claim follows. 

Let $f_1\ldots, f_k\in \A$ be elements of $\A$. We define the functions $F_i=\Psi({f_i})$. We show that if the elements $f_i$ are positive, then the functions $F_i$ are also positive. This follows from the fact that $\Psi$ is a positive operator. More simply,
if the elements $f_i$ are positive, we can find elements $b_1,\ldots, b_k\in \A$ such that $f_i=b_1\cdot b_1^{*}+\dots+b_k\cdot b_k^{*}$. Since $\Psi$ is a $C^*$-algebra homomorphism, we derive $F_i=G_1\cdot \overline{G_1}+\dots+G_k\cdot \overline{G_k}$, where $G_i=\Psi(b_i)$. This confirms our claim.

It remains to establish \eqref{eq_furstenberg system identity0}.
We note that the left-hand side of \eqref{eq_furstenberg system identity0} can be rewritten as \begin{equation*}
    \beta\left(\sigma^{h_1}{f_1}\cdot\ldots\cdot \sigma^{h_{\ell}}{f_{\ell}}\right).
\end{equation*}

The identity ${\Psi}(\sigma( a))=T({\Psi}(a))$ implies that the 
 image of $\sigma^{h_j}f_j$
is the function $T^{h_j}F_j$. 
Since ${\Psi}$ is an algebra homomorphism, it follows that the image of the sequence $\sigma^{h_1}f_1\cdot\ldots\cdot \sigma^{h_{\ell}}f_{\ell}$ is the function $T^{h_1}F_1\cdot\ldots \cdot T^{h_k}F_k$. The conclusion follows from \eqref{eq_characterization of mu}.
\end{proof}

We can now state the main dynamical statement that will be used to establish \cref{Thm_maindensity}.

\begin{Theorem}[Main dynamical theorem]\label{prop_dynamical statement}
   Let $s,k\in \N$, $P\in \Z[x]$ of degree $k$ with $P(0)=0$ and \(\gamma\) be a shift-invariant mean. Assume $(X,\mathcal{X},\mu,T)$ is an invertible ergodic measure preserving system and $(F_m)_{m\in \N}$ is a sequence of measurable functions on $X$ taking values in $[0,1]$ such that the quantity 
\begin{equation}\label{eq_positivity of averages of measurable functions along Bohr sets}
         \limsup_{B\in \mathrm{PB}_0^{\leq k}}\gamma_B\Big(m\mapsto \int F_{m}\, d\mu\big)>0
    \end{equation}
   Then, we have $$\int F_{a}\cdot T^{\frac{1}{a^s} P(\frac{m}{a})}F_a\cdot  F_m\, d\mu> 0.$$ for infinitely many $a,m\in \N$ with $a^{s+1}\mid m$.
\end{Theorem}
\begin{Remark}\label{Remark: counterexample}
    We record here that the ergodicity assumption in this statement is crucial. Specializing to the case $s=1$ and $P(m)=m$, consider the non-ergodic transformation\footnote{This is an adaptation of a construction by Moreira (personal communication).} $T(x,y)=(x,x+y)$ on $\T^2$ with the Lebesgue measure. Take $A=(\frac{1}{4}, \frac{1}{2})\times (\frac{1}{2}, \frac{3}{4})$ and set $F_m(x,y)=\1_{A}(mx,m^2y)$. It is not difficult to prove that $\int F_m\,d\mu=\frac{1}{16}$ for all $m\in \N$, but the function $F_a\cdot T^{\frac{m}{a^2}}F_a\cdot F_m$ is identically zero for all $(a,m)$ with $a^2\mid m$.
    \end{Remark}
\begin{proof}[Proof of \cref{Thm_maindensity} assuming \cref{prop_dynamical statement}]

Fix $s\in \N$, let $E$ and $P\in \Z[x]$ be as in the statement of \cref{Thm_maindensity} and let $\gamma$ and $\beta$ be as in \cref{def_new_density}. 

Let $\A$ be the affinely invariant $C^*$-algebra generated by $\1_{E}$. Applying Lemma~\ref{lem_ergodicity passes to factors}, we deduce that the restriction $\beta|_\A$ to $\A$ is ergodic. It is immediate from the formula for the density in \cref{def_new_density} that $E$ satisfies $$\limsup_{B\in \text{PB}_0^{\leq k}}\gamma_{B}\big(m\mapsto \beta|_{\A}(n\mapsto \,\1_{E}(m^sn))\big)>0,$$
    since the sequence $n\mapsto \1_{E}(m^sn)$ is also an element of $\A$.

We apply Proposition \ref{prop_correspondence principle} for the algebra $\A$ and the mean $\beta|_\A$. We let $(X,\mu, T)$ be the ergodic measure-preserving system and $\Psi\colon \A\to C(X)$ be the $C^{*}$-isomorphism provided by this proposition.
Furthermore, for every $a\in\N$, let $f_a(n)=1_E(a^sn)$ and 
define $F_{a}= \Psi(f_a)$. 
Note that $F_a$ is a continuous function on $X$ taking values in $[0,1]$.

The system provided by \cref{prop_correspondence principle} is not invertible, so iterates involving negative quantities do not necessarily make sense. To overcome this, we pass to the natural extension $(\wt{X},\wt{\mu},\wt{T})$ of $(X,\mu,T)$, where $\wt{X}$ is a compact metric space, $\wt{T}:\wt{X}\to \wt{X}$ is a homeomorphism
and $\wt{\mu}$ is an ergodic $\wt{T}$-invariant Borel probability measure. This is a standard dynamical construction. Any function $F\in C(X)$
taking values in $[0,1]$ has a lift in $C(\wt{X})$ taking values in $[0,1]$. In addition, the integral on the right-hand side of \eqref{eq_furstenberg system identity0} are preserved when replacing the functions by their lifts, $T$ with $\wt{T}$ and $\mu$ by $\mu$. This also allows us to extend \eqref{eq_furstenberg system identity0} to general integer shifts $h_1,\ldots, h_k$ using the shift-invariance of $\beta$. We omit the details.

From now on, we assume that $(X,\mu,T)$ is replaced by the natural extension and the functions $F_a$ have been replaced by their lifts. We keep the same notation for convenience and just assume that $(X,\mu,T)$ is an invertible measure-preserving system.

We deduce from \eqref{eq_furstenberg system identity0} that $$\beta|_\A(n\mapsto 1_E(m^sn))=\int F_m\ d\mu$$ for every $m\in \N$, which implies that $$\limsup_{B\in \text{PB}_0^{\leq k}}\gamma_{B}\Big(m\mapsto \beta\big(n\mapsto \,\1_{E}(m^sn)\big)\Big)=\limsup_{B\in \text{PB}_0^{\leq k}}\gamma_B\left(m\mapsto \int F_m\, d\mu\right).$$

Our assumption on the positivity of the left-hand side implies that 
the sequence $F_m$ of continuous functions and the mean $\gamma$ satisfy \eqref{eq_positivity of averages of measurable functions along Bohr sets}.
It follows that there exist infinitely many $a,m\in \N$ with $a^{s+1}\mid m$
such that $$\int F_{a}\cdot T^{\frac{1}{a^s} P(\frac{m}{a})}    F_a\cdot  F_m\, d\mu> 0.$$
However, using the identity \eqref{eq_furstenberg system identity0} and the fact that $F_m$ is the image of $n\mapsto 1_E(m^sn)$, we deduce that \begin{equation*}
    \beta\left(n\mapsto \1_{E}(a^sn)\1_E\left(a^s\left(n+\tfrac{P(m/a)}{a^s}\right)\right)\1_{E}(m^sn)\right)=\int F_{a}\cdot  T^{\frac{1}{a^s} P(\frac{m}{a})}  F_a\cdot  F_m\, d\mu.
\end{equation*}
Therefore, we can find infinitely many triples $(a,m,n)$ with $a^{s+1}\mid m$ so that $$\1_{E}(a^sn)\1_E\left(a^sn+P(\tfrac{m}{a})  \right)\1_{E}(m^sn)>0.$$
Choosing $x=a^sn$ and $y=m/a$, we deduce that $\{x,x+P(y),xy^s\}\subset E$.
\end{proof}
It remains to establish \cref{prop_dynamical statement}, which is the content of the following section.

\section{Proof of the dynamical statement}

In this section, we establish \cref{prop_dynamical statement}. More precisely, we prove the following more technical result, from which \cref{prop_dynamical statement} follows directly. 

\begin{Proposition}\label{prop_dynamical statement full}
    Let $k,s\in \N$, $P\in \Z[x]$ of degree $k$ with $P(0)=0$ and \(\gamma\) be a shift-invariant mean. Assume $(X,\mathcal{X},\mu,T)$ is an invertible ergodic measure-preserving system and $(F_m)_{m\in \N}$ is a sequence of measurable functions on $X$ taking values in $[0,1]$ such that the quantity 
    \begin{equation}\label{eq_positivity of averages of measurable functions along Bohr sets in full}
          \delta:= \limsup_{B\in \mathrm{PB}_0^{\leq k}}\gamma_B\Big(m\to \int F_{m}\, d\mu\big)
    \end{equation}
    is positive. Then, for every $\e>0$, there exists a sequence of polynomial  Bohr$_0^{\leq k}$ sets $(B_{\ell})_{\ell\in \N}$ and a \Folner{} sequence $(\Phi_M)_{M\in\N}$ such that:

\begin{enumerate}[label=(\roman{enumi}),ref=(\roman{enumi})]
    \item\label{itm_i} $a^{s+1}\mid m$ whenever $a\leq \ell_1$ and $m\in B_{\ell_1}$,
    \item\label{itm_ii}
    we have \[ 
   \lim_{\ell_2\to\infty}
    \lim_{M_2\to\infty}
    \lim_{\ell_1\to\infty}
    \lim_{M_1\to\infty} \E_{a\in B_{\ell_2}\cap \Phi_{M_2}}  \E_{m\in B_{\ell_1}\cap \Phi_{M_1}}\int F_a\cdot T^{\frac{1}{a^s} P(\frac{m}{a})}F_a\cdot F_m\,d\mu\geq \delta^4-\epsilon,
\]  where all the limits in the previous expression exist.
\end{enumerate}
In particular, $$\int F_{a}\cdot T^{\frac{1}{a^s} P(\frac{m}{a})} F_a\cdot  F_m\, d\mu\geq \delta^4-\e$$ for infinitely many $a,m\in \N$ with $a^{s+1}\mid m$.
\end{Proposition}
\begin{Remark}
The existence of the limits in part~\ref{itm_ii} is only a technicality, as we can always pass to subsequences of $(\Phi_M)_{M\in\N}$ and $(B_{\ell})_{\ell\in\N}$ such that these limits exist.
We also note that when $a^{s+1}\mid m$, the fraction $\frac{1}{a^s}P(\frac{m}{a})$ is an integer since $P(0)=0$.
Therefore, by part~\ref{itm_i}, the expression $ T^{\frac{1}{a^s} P(\frac{m}{a})}$ is well-defined because the order of the limits allows us to restrict to $\ell_1\geq a$.
\end{Remark}

\begin{proof}

Our first objective is to construct the sequence of sets $(B_{\ell})_{\ell\in \N}$.
For every $m\in \N$, we use Corollary \ref{cor_decomposition via Kronecker} to write
$$F_{m}=F_{m,\mathrm{str}}+F_{m,\mathrm{unf}}$$
for functions $F_{m,\mathrm{str}},F_{m,\mathrm{unf}}\in L^{\infty}(\mu)$, such that $F_{m,\mathrm{str}}\in L^{\infty}(\mu)$ is a compact function and $F_{m,\mathrm{unf}}$ is a weak-mixing function.

For every $\ell\in\N$, we can find a polynomial Bohr$_0$ set $C_{\ell}\subset\N$, such that  for all $a\in\{1,\ldots,\ell\}$
\begin{align}
a^{s+1}\mid m,\qquad \text{for all } m\in C_{\ell}
\label{eqn_bohr_set_divisibility_assumption}\\
\sup_{m\in C_{\ell}} 
\int \big|F_{a,\mathrm{str}}\big(T^{\frac{1}{a^s}P(\frac{m}{a})}x)-  F_{a,\mathrm{str}}(x)\big|^2 d\mu(x)\ \leq \frac{1}{\ell}\label{eqn_bohr_set_of_almost_periods}.
\end{align}
To see why this is true, note first that the set of $m$ that are divisible by $(\ell!)^{s+1}$ is a PB$_0^{\leq k}$ set, and all elements of this set satisfy the first condition. For the second assertion, note that for every $a\in \N$, there exist $t_a\in \N$, eigenfunctions $g_{a,1},\ldots, g_{a,t_a}\in L^2(\mu)$, and coefficients $c_{a,1},\ldots, c_{a,t_a}$, such that \begin{equation}\label{eq_approximation of structured by eigenfunctions}
\Big\|F_{a,\mathrm{str}}-\sum_{t\leq t_a}c_{a,t}g_{a,t}\Big\|_{L^2(\mu)}\leq \frac{1}{3\ell}.
\end{equation}Additionally, let $\lambda_{a,t}\in \C$ denote the eigenvalue of the function $g_{a,t}$.
For any $a\in \N$, we let $Q_a(m)=a^{k}P(m/a)$, which is a polynomial with integer coefficients since $P$ has degree $k$.

Now, we define the PB$_0^{\leq k}$ set (we make an arbitrary choice for the root $\lambda_{a,t}^{1/a^{k+s}}$ here)
$${C}_{a,\ell}=\left\{m\in \N\colon \Big|\Big(\lambda_{a,t}^{\frac{1}{a^{k+s}}}\Big)^{Q_a(m)}-1\Big|\leq \frac{1}{3\ell\sum_{t\leq t_a}|c_{a,t}|} \ \text{for all } t\in [t_a]\right\}\cap (\ell!)^{s+1}\N.$$
Observe that since $Q_a$ takes integer values, we have
$\big(\lambda_{a,t}^{\frac{1}{a^{k+s}}}\big)^{Q_a(m)}=\lambda_{a,t}^{\frac{Q_a(m)}{a^{k+s}}}$.
We infer that for any $m\in C_{a,\ell}$, we have 
\begin{equation}\label{eq_Bohr property of eigenvalues}
    \big|\lambda_{a,t}^{\frac{1}{a^s} P(\frac{m}{a})}-1\big|\leq\frac{1}{3\ell\sum_{t\leq t_a}|c_{a,t}|}
\end{equation}   for all admissible values of $t$.
The triangle inequality implies that for any $m\in {C}_{a,\ell}$, we have 
\begin{align*}
\big\| & T^{\frac{1}{a^s}P(\frac{m}{a})} F^{}_{a,\mathrm{str}}-  F^{}_{a,\mathrm{str}}\big\|_{L^2(\mu)}\\
&\leq \Big\|T^{\frac{1}{a^s}P(\frac{m}{a})}F_{a,\mathrm{str}}-\sum_{t\leq t_a}c_{a,t}T^{\frac{1}{a^s}P(\frac{m}{a})}g_{a,t}\Big\|_{L^2(\mu)}+\Big\|F_{a,\mathrm{str}}-\sum_{t\leq t_a}c_{a,t}g_{a,t}\Big\|_{L^2(\mu)}  \\
&\qquad \qquad \qquad  \qquad  \qquad \qquad \ \ \ \ + \Big\|\sum_{t\leq t_a} c_{a,t}T^{\frac{1}{a^s}P(\frac{m}{a})}g_{a,t}-\sum_{t\leq t_a}c_{a,t}g_{a,t}\Big\|_{L^2(\mu)}\\
&\leq \frac{2}{3\ell}+
\sum_{t\leq t_a}|c_{a,t}|\cdot\|\lambda_{a,t}^{\frac{1}{a^s}P(\frac{m}{a})} g_{a,t}-g_{a,t}\|_{L^2(\mu)}\ \\
&\leq 3\cdot \frac{1}{3\ell}=\frac{1}{\ell}.
\end{align*}
In these inequalities, we used \eqref{eq_approximation of structured by eigenfunctions} for the first two terms (note that $T$ is an isometry on $L^2(\mu)$), and for the third term we used that $g_{a,t}$ are eigenfunctions and \eqref{eq_Bohr property of eigenvalues}. 

Finally, we need the last relation to hold for all $a\in [\ell]$, so we  consider the intersection \begin{equation*}
   C_{\ell}=\bigcap_{a=1}^{\ell} C_{a,\ell}\cap (\ell!)^{s+1}\Z.
\end{equation*}
It follows that the set $C_{\ell}\in \text{PB}_0^{\leq k}$ satisfies the desired properties
\eqref{eqn_bohr_set_divisibility_assumption} and \eqref{eqn_bohr_set_of_almost_periods}.

Now fix $\e>0$.
By \cref{D: limsup for filters}, we deduce that for all sets $C_{\ell}$ above we have 
\begin{equation*}
    \sup_{B\in \mathrm{PB}_0^{\leq k}} \gamma_{B\cap C_{\ell}}\Big(m\mapsto \int F_m\, d\mu\Big)\geq \delta.
\end{equation*}
This implies that we can find a sequence
$B_{\ell}\subseteq C_{\ell}$ of sets in the family PB$_0^{\leq k}$, such that \begin{equation}\label{eq_lower bound by delta/2 with mean}
\frac{\gamma\Big(m\mapsto \1_{B_{\ell}}(m)\int F_m\, d\mu \Big)}{\gamma(B_{\ell})}>\delta-\e.
\end{equation}for every $\ell\in \N$. This completes the construction of the sequence $(B_\ell)_{\ell\in\N}$.

Our next goal is to construct the \Folner\ sequence $(\Phi_M)_{M\in\N}$. 
Let $\B$ denote the shift-invariant $C^*$-algebra generated by 
\begin{enumerate}
    \item $\1_{B_{\ell}}$ for $\ell\in \N$ and
    \item  the sequence $m\mapsto \int F_m \,d\mu$.
\end{enumerate} 
We apply Lemma~\ref{L: mean on separable is Folner} for the mean $\gamma$ and the separable $C^*$-algebra $\B$ to obtain a \Folner\ sequence $(\Phi_{M})_{M\in \N}$, such that \begin{equation*}
    \gamma(m\mapsto f(m))=\lim_{M\to+\infty}\E_{m\in \Phi_M} f(m)\ \text{for every } f\in \B.
\end{equation*}
Since $\1_{B_{\ell}}$ and $m\mapsto \int F_m\, d\mu$ are elements of $\B$, we get
\begin{equation*}
    \frac{\gamma\left(m\mapsto \1_{B_{\ell}}(m)\int F_m \, d\mu\right)}{\gamma(B_{\ell})}=\frac{1}{\gamma(B_{\ell})}\lim_{M\to+\infty}\E_{m\in \Phi_M} \1_{B_{\ell}}(m)\int F_m\, d\mu.
\end{equation*}
Thus, \eqref{eq_lower bound by delta/2 with mean} implies \begin{equation}\label{eq_lower bound by delta/2}
   \lim_{M\to+\infty} \E_{m\in \Phi_{M}\cap B_{\ell}} \int F_m \,d\mu> \delta-\e
\end{equation}for every $\ell\in \N$.
This completes the construction of the \Folner{} sequence $(\Phi_M)_{M\in\N}$.

Note that by passing to a subsequence of $B_{\ell}$ and $\Phi_M$ if necessary (which doesn't affect the validity of \eqref{eq_lower bound by delta/2}), we may assume that the iterated limits \begin{align*}
    L:=
 &\lim_{\ell_2\to\infty}
\lim_{M_2\to\infty}
\lim_{\ell_1\to\infty}
\lim_{M_1\to\infty} \E_{a\in B_{\ell_2}\cap \Phi_{M_2}}  \E_{m\in B_{\ell_1}\cap \Phi_{M_1}}\ \int F_a\cdot T^{\frac{1}{a^s} P(\frac{m}{a})}F_a\cdot F_m\,d\mu
\end{align*}exist.
To make the presentation cleaner from now on, we will use the notation \begin{equation*}
    \lim_{(n_1,\ldots, n_k)\to\infty}a_{n_1,\ldots,n_k}
\end{equation*}for the sequence of iterated limits \begin{equation*}
    \lim_{n_1\to+\infty}\cdots \lim_{n_k\to+\infty}a_{n_1,\ldots,n_k}
\end{equation*}for a given multiply indexed sequence $a_{n_1,\ldots,n_k}:\N^k\to \C$, assuming that all limits in the previous expression exist.

We now apply the decomposition $F=F_{a,\mathrm{str}}+F_{a,\mathrm{unf}}$ to rewrite our averages as
\begin{align*}
L= \lim_{(\ell_2,M_2,\ell_1,M_1)\to\infty} \E_{a\in B_{\ell_2}\cap \Phi_{M_2}}&  \E_{m\in B_{\ell_1}\cap \Phi_{M_1}}
\int F_a\cdot T^{\frac{1}{a^s} P(\frac{m}{a})}F_a   \cdot F_m\, d\mu
\\
& =\Sigma_1+\Sigma_2,
\end{align*}
where
\begin{align*}
\Sigma_1&=  \lim_{(\ell_2,M_2,\ell_1,M_1)\to\infty} &
\E_{a\in B_{\ell_2}\cap \Phi_{M_2}}  \E_{m\in B_{\ell_1}\cap\Phi_{M_1}}
\int F_a\cdot T^{\frac{1}{a^s} P(\frac{m}{a})}F_{a,\mathrm{str}}\cdot F_m\, d\mu,
\\
\Sigma_2&=  \lim_{(\ell_2,M_2,\ell_1,M_1)\to\infty} &
\E_{a\in B_{\ell_2}\cap \Phi_{M_2}}  \E_{m\in B_{\ell_1}\cap \Phi_{M_1}}
 \int F_a \cdot T^{\frac{1}{a^s} P(\frac{m}{a})}F_{a,\mathrm{unf}}\cdot  F_m\,d\mu.
\end{align*}Here, we can pass to subsequences again if necessary, to ensure that the new iterated limits exist as well. 

We claim that $\Sigma_2=0$. 
Using the Cauchy-Schwarz inequality, we have
\begin{align*}
\Big|\E_{a\in B_{\ell_2}\cap  \Phi_{M_2}} &\E_{m\in B_{\ell_1}\cap \Phi_{M_1}} 
 \int F_a \cdot T^{\frac{1}{a^s} P(\frac{m}{a})}F_{a,\mathrm{unf}}\cdot  F_m\,d\mu \Big|^2
\\
&\leq \E_{m\in B_{\ell_1}\cap \Phi_{M_1}}
\int  \Big|\E_{a\in B_{\ell_2}\cap \Phi_{M_2}} F_a\cdot T^{\frac{1}{a^s} P(\frac{m}{a})} F_{a,\mathrm{unf}} \Big|^2\,d\mu
\\
&=\E_{a,b\in B_{\ell_2}\cap\Phi_{M_2}} \E_{m\in B_{\ell_1}\cap \Phi_{M_1}}
\int  F_a\cdot F_b \cdot T^{\frac{1}{a^s} P(\frac{m}{a})}F_{a,\mathrm{unf}}\cdot 
T^{\frac{1}{b^s} P(\frac{m}{b})}F_{b,\mathrm{unf}} \, d\mu.
\end{align*}

Note that the contribution of the diagonal terms $a=b$ is negligible, so we may assume throughout $a\neq b$.

For large enough $\ell_1$, we have $(ab)^{s+1}$ divides $m$ by the construction of the Bohr sets. Inserting the weight $1_{(ab)^{s+1}\mid m}$ into the average and doing the change of variables $m\to (ab)^{s+1}m$, it suffices to show that 
\begin{equation*}
   \limsup_{M\to\infty}\E_{m\in B_{\ell_1}/(ab)^{s+1}\cap \Phi_{M}/(ab)^{s+1}} \int  F_a\cdot F_b \cdot T^{P_a(m)}F_{a,\mathrm{unf}}\cdot 
T^{P_b(m)}F_{b,\mathrm{unf}} \, d\mu
\end{equation*}
converges to zero for every $\ell_1\in \N$, where $P_a,P_b$ are polynomials with integer coefficients defined by $P_a(m)=\frac{1}{a^s}P(a^{s}b^{s+1}m)$ and $P_b(m)=\frac{1}{b^s}P(a^{s+1}b^{s}m)$.

Since $B_{\ell_1}$ is a PB$_0^{\leq k}$ set, a standard approximation argument shows that it suffices to establish the estimate \begin{equation}\label{eq_final average with weak mixing}
    \limsup_{M\to\infty}\left|\E_{m\in \frac{\Phi_{M}}{(ab)^{s+1}}}  e\left( Q(m)\right)\int  H \cdot T^{P_a(m)}F_{a,\mathrm{unf}}
T^{P_b(m)}F_{b,\mathrm{unf}}\, d\mu\right|=0
\end{equation}for any function $H\in L^{\infty}(\mu)$, $\theta\in \T$ and any polynomial $Q$ with real coefficients and degree at most $k$.
Notice now that the assumptions of Proposition \ref{prop_optimal factor for 2 polynomials} are satisfied (since $\Phi_{M}/(ab)^{s+1}$ forms a \Folner\ sequence, $P_a, P_b$ are essentially distinct as they have distinct leading coefficients and the degree assumption is satisfied), and thus, \eqref{eq_final average with weak mixing} follows.

To conclude the proof, we show that the remaining term $\Sigma_1$ yields the desired lower bound. For this, recall that the original functions $F_a$ are non-negative, which also implies that the same holds for the functions $F_{a,\mathrm{str}}$. In particular, by passing to subsequences, if necessary, for the limit below to exist, we have the lower bound 
\begin{multline*}
    L=\Sigma_1\geq\\
    \lim_{(\ell_2,M_2,\ell_1,M_1)\to\infty} 
\E_{a\in B_{\ell_2}\cap \Phi_{M_2}}  \E_{m\in B_{\ell_1}\cap \Phi_{M_1}}
\int F_a\cdot T^{\frac{1}{a^s} P(\frac{m}{a})} F_{a,\mathrm{str}}\cdot  F_m\cdot F_{m, \mathrm{str}}\, d\mu.
\end{multline*}
Using \eqref{eqn_bohr_set_of_almost_periods} for every $m\in B_{\ell_1}\subseteq C_{\ell_1}$ and noticing that the variable $m$ is very large compared to the numbers $a\in \Phi_{M_2}$, we have
\begin{equation}\label{eq_only structured parts remaining}
  L\geq  \lim_{(\ell_2,M_2,\ell_1,M_1)\to\infty}
\E_{a\in B_{\ell_2}\cap \Phi_{M_2}}  \E_{m\in B_{\ell_1}\cap \Phi_{M_1}}\int 
F_a\cdot F_{a,\mathrm{str}} \cdot F_m\cdot F_{m, \mathrm{str}}\, d\mu.
\end{equation}

Next, consider the function
\[
\varphi(\ell_2,M_2,\ell_1,M_1)=
\E_{a\in B_{\ell_2}\cap \Phi_{M_2}}  \E_{m\in B_{\ell_1}\cap \Phi_{M_1}}\int 
F_a\cdot F_{a,\mathrm{str}} \cdot F_m\cdot F_{m, \mathrm{str}}\, d\mu.
\]

Using Ramsey's theorem for $4$-sets and replacing $(B_\ell)_{\ell\in\N}$ and $(M_j)_{j\in\N}$ with subsequences of themselves if necessary, we can assume that there exists some $C_1\in[0,1]$ such that\begin{equation}\label{eq_first ramsey aprroximation}
    |\varphi(\ell_2,M_2,\ell_1,M_1)-C_1|\leq\epsilon
\end{equation}
for all $M_1,\ell_1,M_2,\ell_2$ with $M_1>\ell_1>M_2>\ell_2$.
Using Ramsey's theorem for $2$-sets and replacing $(B_\ell)_{\ell\in\N}$ and $(M_j)_{j\in\N}$ by subsequences once more, we can also assume without loss of generality that there exists some $C_2\in[0,1]$ such that
\begin{equation}\label{eq_second ramsey approximation}
    \Big|\E_{m\in B_{\ell}\cap \Phi_M}  
 \int F_m\cdot F_{m, \mathrm{str}}\, d\mu -C_2\Big|\leq\epsilon
\end{equation}
for all $M,\ell$ with $M>\ell$. 
Using the Cauchy-Schwarz inequality, we obtain for any $Q\in \N$
\begin{align*}
C_2^2&=
\Big( \E_{q\in [Q]}\E_{\ell\in(2^q,2^{q+1}]}\E_{M\in(2^{q+1},2^{q+2}]}  \E_{m\in B_{\ell}\cap \Phi_{M}}  
 \int F_m\cdot F_{m, \mathrm{str}}\, d\mu \Big)^2+\Oh(\epsilon)
 \\
 &\leq
\int 
\Big(\E_{q\in [Q]}\E_{\ell\in(2^q,2^{q+1}]}\E_{M\in(2^{q+1},2^{q+2}]}  \E_{m\in B_{\ell}\cap \Phi_M}\     
   F_m\cdot F_{ m,\mathrm{str}}\Big)^2\,d\mu+\Oh(\epsilon)
   \\&=
\frac{2}{Q^2}\sum_{\substack{p,q\in [Q]\\ p+1<q}}
\E_{\substack{\ell_2\in(2^p,2^{p+1}]\\ M_2\in(2^{p+1},2^{p+2}]}}
\E_{\substack{\ell_1\in(2^q,2^{q+1}]\\ M_1\in(2^{q+1},2^{q+2}]}}\,
\varphi(\ell_2,M_2,\ell_1,M_1)+\Oh\left(\epsilon+\frac{1}{Q}\right)
\\
&=
C_1+\Oh\left(\epsilon+\frac{1}{Q}\right).
\end{align*}
Taking $Q$ large enough compared to $\e^{-1}$ and combining \eqref{eq_first ramsey aprroximation} and \eqref{eq_second ramsey approximation}, this proves that
\begin{align*}
L&\geq \lim_{(\ell_2,M_2,\ell_1,M_1)\to\infty}
\E_{a\in B_{\ell_2}\cap \Phi_{M_2}}  \E_{m\in B_{\ell_1}\cap \Phi_{M_1}}\int 
F_a\cdot F_{a,\mathrm{str}} \cdot F_m\cdot F_{m, \mathrm{str}}\, d\mu
\\
& \geq \Big(\lim_{(\ell,M)\to\infty}
\E_{m\in B_{\ell}\cap \Phi_{M}}\int F_m\cdot F_{m, \mathrm{str}}\, d\mu
\Big)^2+\Oh(\epsilon).
\end{align*}Now note that if we decompose $F_m$ into its structured and uniform parts, we will have $$\int F_{m,\mathrm{unf}}\cdot F_{m,\mathrm{str}}\, d\mu=0$$by orthogonality. We deduce that 
\begin{align*}
L&\geq \Big(\lim_{(\ell,M)\to\infty}\E_{m\in B_{\ell}\cap \Phi_{M}}\int F_m\cdot F_{m, \mathrm{str}}\, d\mu
\Big)^2+\Oh(\epsilon)
\\
& \geq \Big(\lim_{(\ell,M)\to\infty}
\E_{m\in B_{\ell}\cap\Phi_{M}}\int F_{m, \mathrm{str}}^2\, d\mu
\Big)^2+\Oh(\epsilon)\\
&\geq \Big(\lim_{(\ell,M)\to\infty}
\E_{m\in B_{\ell}\cap \Phi_{M}}\int F_{m, \mathrm{str}}\, d\mu
\Big)^4+\Oh(\epsilon),
\end{align*}where the last inequality is proven via iterated applications of Cauchy-Schwarz. Now, we observe that $\int F_m\,\,d\mu=\int F_{m,\mathrm{str}}\, d\mu$, which follows from the fact that the weak-mixing part of any bounded measurable function has integral equal to zero (as it is orthogonal to the constant functions which are $T$-invariant). Combining this with the lower bound in \eqref{eq_lower bound by delta/2}, we conclude that   
\begin{align*}
L
\geq 
\Big(\lim_{(\ell,M)\to\infty}
\E_{m\in B_{\ell}\cap \Phi_{M}}\int F_{m}\, d\mu
\Big)^4+\Oh(\e)\geq \delta^4+\Oh(\e).
\end{align*}
Adjusting $\e$ accordingly, the claimed result follows.
\end{proof}

\bibliographystyle{aomalphanomr}
\bibliography{references}




\bigskip
\footnotesize
\noindent
Florian K.\ Richter\\
\textsc{{\'E}cole Polytechnique F{\'e}d{\'e}rale de Lausanne (EPFL)}\\
\href{mailto:f.richter@epfl.ch}
{\texttt{f.richter@epfl.ch}}

\bigskip
\footnotesize
\noindent
Konstantinos Tsinas\\
\textsc{{\'E}cole Polytechnique F{\'e}d{\'e}rale de Lausanne (EPFL)}\\
\href{mailto:konstantinos.tsinas@epfl.ch}
{\texttt{konstantinos.tsinas@epfl.ch}}

\end{document}